\documentclass{amsart}
\usepackage{amsfonts}
\usepackage{graphicx}
\usepackage{amscd}
\usepackage{amsmath}
\usepackage{amssymb}
\usepackage{stmaryrd} 

\makeatletter
\@namedef{subjclassname@2010}{%
  \textup{2010} Mathematics Subject Classification}
\makeatother

\theoremstyle{plain}
\newtheorem*{acknowledgement}{Acknowledgement}

\newtheorem{corollary}{\bf Corollary}
\newtheorem{definition}{\bf Definition}
\newtheorem{lemma}{\bf Lemma}

\newtheorem{remark}{Remark}

\newtheorem{theorem}{\bf Theorem}

\theoremstyle{definition}

\numberwithin{equation}{section}

\title[Critical metrics of the volume functional]{Critical metrics of the volume functional with pinched curvature}

\author{H. Baltazar}

\author{C. Queiroz}

\address[H. Baltazar]{Departamento de Matem\'{a}tica, Universidade Federal do Piau\'{\i}\\
64049-550 Te\-re\-si\-na, Piau\'{\i}, Brazil.}
\email{halyson@ufpi.edu.br}

\address[C. Queiroz]{Departamento de Matem\'{a}tica, Universidade Federal do Piau\'{\i}\\
64049-550 Te\-re\-si\-na, Piau\'{\i}, Brazil.}
\email{chrisqueiroz@ufpi.edu.br}

\subjclass[2010]{Primary 53C25, 53C20, 53C21; Secondary 53C65}
\keywords{Volume Functional; critical metrics; Weyl tensor}

\begin{document}

\newcommand{\spacing}[1]{\renewcommand{\baselinestretch}{#1}\large\normalsize}
\spacing{1.2}

\begin{abstract}
In this paper, we prove that a critical metric of the volume functional with pinched Weyl curvature is isometric to a geodesic ball in $\mathbb{S}^{n}.$ Moreover, we provide a necessary and sufficient condition on the norm of the gradient of the potential function in order to classify such critical metrics.
\end{abstract}

\maketitle

\section{Introduction}\label{intro}

The investigation of critical metrics becomes a classical topic in geometric analysis. In this sense, much attention has been given to the critical metrics of the volume functional $V:\mathcal{M}_{\gamma}^{R}\rightarrow\mathbb{R}$, where $\mathcal{M}_{\gamma}^{R}$ denotes the space all structures on $M$ with constant scalar curvature and prescribed boundary metric $\gamma$. In \cite{MT09}, inspired by ideas developed in \cite{FST09}, Miao and Tam provided sufficient conditions for a metric $g\in\mathcal{M}_{\gamma}^{R}$ to be a critical point. In fact, they proved that $g$ is a critical point of the functional $V|_{\mathcal{M}_{\gamma}^{R}}$ if and only if  there exist a smooth function $f$ on $M$ such that $f|_{\partial M}=0$ and satisfies the following system of PDE's
\begin{equation}\label{eqMiaoTam}
-(\Delta f) g+Hessf-fRic=g,
\end{equation}
provided that the first Dirichlet eigenvalue of $(n-1)\Delta_{g}+R$ is positive. Here, the terms $Ric$ and $Hessf$ stands for the Ricci tensor and Hessian form associated to $g$ on $M$, respectively. 

Following the terminology used in \cite{BDR,MT09,MT11} we recall the definition of such critical metrics investigated by Miao-Tam, which for simplicity will be called Miao-Tam critical metrics.

\begin{definition}\label{def1} 
A Miao-Tam critical metric is a 3-tuple $(M^n,\,g,\,f),$ where $(M^{n},\,g),$ $n\geq3,$ is a compact oriented Riemannian  manifold with a smooth boundary $\partial M$ and $f:M^{n}\rightarrow\mathbb{R}$ is a smooth function such that $f^{-1}(0)=\partial M$ and satisfies (\ref{eqMiaoTam}).
\end{definition}

There are several works about classification of Miao-Tam critical metrics in\-vol\-ving the Weyl curvature tensor, see for instance, \cite{BBBV20,bbb,BDR21,balt18,BDR,MT11}. Hence, motivated by \cite{Catino13} and based in techniques developed in \cite{BaltCPE,Balt20}, we shall provide a classification of a such critical metrics considering a pointwise pinching condition. 
More precisely, we have the following result.

\begin{theorem}\label{Wineq}
Let $(M^{n},g,f)$ be a Miao-Tam critical metric with nonnegative scalar curvature and satisfying
\begin{equation}\label{Welyestimate}
\left|W+\frac{\sqrt{n}}{\sqrt{2}(n-2)}\mathring{Ric}\varowedge g\right|\leq\frac{R}{\sqrt{2(n-1)(n-2)}}.
\end{equation}
Then $(M^{n},g)$ is isometric to a geodesic ball in $\mathbb{S}^{n}.$ 
\end{theorem}
The symbol $\varowedge$, in the above expression, denotes the Kulkarni-Nomizu product which is defined for any two symmetric $(0,2)$-tensor $S$ and $T$ as follows
$$(S\varowedge T)_{ijkl}=S_{ik}T_{jl}+S_{jl}T_{ik}-S_{il}T_{jk}-S_{jk}T_{il}.$$

Since the Weyl tensor vanishes in dimension three, we obtain the following consequence:
\begin{corollary}
Let $(M^{3},g,f)$ be a three dimensional Miao-Tam critical metric with nonnegative scalar curvature and satisfying
\begin{equation}\label{3dcase}
|\mathring{Ric}|\leq\frac{R}{\sqrt{24}}.
\end{equation}
Then $(M^{3},g)$ is isometric to a geodesic ball in $\mathbb{S}^{3}.$ 
\end{corollary}

\begin{remark}
It is easy to verify that, the pinching (\ref{3dcase}) becomes in this case
$$|\mathring{Ric}|\leq\frac{R}{\sqrt{24}}\leq\frac{R}{\sqrt{6}}.$$
Then it suffices to apply \cite[Corollary 1.5]{balt18} to conclude that $(M^{3},g)$ is isometric to a geodesic ball in $\mathbb{S}^{3}.$
\end{remark}

In order to justify our second result will be necessary to remember the classical simple connected examples of Miao-Tam critical metrics and a property involving its potential function. The first one is the Euclidean ball of radius $r_{0}$ in $\mathbb{R}^{n}$ with standard metric $g_{0}$ and potential function
$$f(x)=\frac{1}{2(n-1)}(r_{0}^{2}-|x|^{2}).$$
Now, since the scalar curvature is null and $|\nabla f |_{|_{\partial M}}=\frac{r_{0}}{n-1}$ we can deduce 
$$|\nabla f|^{2}+\frac{Rf^{2}}{n(n-1)}+\frac{2f}{n-1}=\frac{r_{0}^{2}}{(n-1)^{2}}=|\nabla f|_{|_{\partial M}}^{2}.$$
At the same time, it is not hard to verify that $(\mathbb{S}^{n},g_{0}),$ the standard sphere with canonical metric $g_{0},$ is a Miao-Tam critical metric with potential function
$$f(x)=\frac{1}{n-1}\left(\frac{cosr}{cos r_{0}}-1\right),$$
where $r$ is the geodesic distance from the point $(0,\ldots,1)$ and radius $r_{0}\neq\frac{\pi}{2}.$
Furthermore, with a straightforward computation we obtain the same identity, 
$$|\nabla f|^{2}+\frac{Rf^{2}}{n(n-1)}+\frac{2f}{n-1}=\frac{sen^{2} r_{0}}{(n-1)^{2}cos^{2}r_{0}}=|\nabla f|_{|_{\partial M}}^{2}.$$
Finally, the last interesting example is the hyperbolic space $\mathbb{H}^{n}$ embedded in the well-known Minkowski space ($\mathbb{R}^{n+1},dx_{1}^{2}+\ldots+dx_{n}^{2}-dt^{2}$),
$$\mathbb{H}^{n}=\{(x_{1},\ldots,x_{n},t)\;|\;x_{1}^{2}+\ldots x_{n}^{2}-t^{2}=-1,\;t>0\}.$$
The geodesic ball in $\mathbb{H}^{n}$ centred at $(0,\ldots,0,1)$ is a Miao-Tam critical metric with potential function
$$f(x)=\frac{1}{n-1}\left(1-\frac{coshr}{cosh r_{0}}\right).$$
Similarly, the potential function must satisfy the following expression
$$|\nabla f|^{2}+\frac{Rf^{2}}{n(n-1)}+\frac{2f}{n-1}=\frac{senh^{2} r_{0}}{(n-1)^{2}cosh^{2}r_{0}}=|\nabla f|_{|_{\partial M}}^{2}.$$

It is natural to ask whether these examples are the only Miao–Tam critical metrics that satisfy such a property. In this sense, inspired by ideas developed by Leandro \cite{L15} we provide a complete answer to this question. More precisely, we prove the following result.

\begin{theorem}\label{MTf}
Let $(M^{n},g,f),$  be a Miao-Tam critical metric satisfying 
\begin{equation}\label{idenTHM2}
|\nabla f|^{2}+\frac{Rf^{2}}{n(n-1)}+\frac{2f}{n-1}=|\nabla f|_{|_{\partial M}}^{2}.
\end{equation}
Then $M^{n}$ is isometric to a geodesic ball in a simply connected space form $\mathbb{R}^{n},$  $\mathbb{H}^{n}$ or $\mathbb{S}^{n}.$
\end{theorem}

\begin{remark}
Let us point out that, for equality (\ref{idenTHM2}) to make sense, it is necessary to remember that a Miao-Tam critical metric has $|\nabla f|$ constant along $\partial M,$ see \cite[Section 3]{BDRR} for more details. 
\end{remark}

\section{Preliminaries}
\label{Preliminaries}

In this section we need recall some special tensors which will be important for understanding the desired results. We start with Weyl tensor, given by the decomposition formula
\begin{equation}\label{weyl}
W_{ijkl}=R_{ijkl}-\frac{1}{n-2}(Ric\varowedge g)_{ijkl}+\frac{R}{2(n-1)(n-2)}(g\varowedge g)_{ijkl},
\end{equation}
where $R_{ijkl}$ stands for the Riemann curvature tensor. In the sequel, we have the Cotton tensor $C$ which is given by
\begin{equation}\label{cotton}
\displaystyle{C_{ijk}=\nabla_{i}R_{jk}-\nabla_{j}R_{ik}-\frac{1}{2(n-1)}\big(\nabla_{i}Rg_{jk}-\nabla_{j}R g_{ik}).}
\end{equation}
It is easy to check that $C_{ijk}$ is skew-symmetric in the first two indices and trace-free in any two indices. These two tensors described above are related as follows
\begin{equation}\label{cottonwyel}
\displaystyle{C_{ijk}=-\frac{(n-2)}{(n-3)}\nabla_{l}W_{ijkl},}
\end{equation}provided $n\ge 4.$

Now, we note that the fundamental equation of a Miao-Tam critical metric can be rewritten in the tensorial language as follows
\begin{equation}\label{eq:tensorial}
-\Delta fg_{ij}+\nabla_{i}\nabla_{j}f -fR_{ij}=g_{ij}.
\end{equation}
Tracing (\ref{eq:tensorial}) we have
\begin{equation}\label{eqtrace}
\Delta f=-\frac{R}{n-1}f-\frac{n}{n-1}.
\end{equation}
Furthermore, by using (\ref{eqtrace}) it is not difficult to check that
\begin{equation}\label{IdRicHess} 
f\mathring{Ric}=\mathring{Hess f},
\end{equation}
where $\mathring{T}$ stands for the traceless of $T.$

Under this notation we get the following formula for a Miao-Tam critical metric
\begin{equation}\label{auxC}
\nabla_{j}fR_{ik}-\nabla_{i}fR_{jk}=fC_{ijk}-R_{ijkl}\nabla_{l}f-\frac{R}{n-1}(\nabla_{i}fg_{jk}-\nabla_{j}fg_{ik}).
\end{equation} 
Its proof can be found in \cite{BDR}.

\section{Key lemmas}
In this section we shall deduce a couple of integral identities,  which allows us to obtain the classification of the Miao-Tam critical metrics under a curvature estimate. Both integrals are direct consequences of the Bochner type formulas obtained by first author and Ribeiro Jr. in \cite{balt18}, see Lemma 3.1 and Lemma 3.2 for more details.

\begin{lemma}\label{auxint1}
Let $(M^{n},g,f)$ be a Miao-Tam critical metric. Then we have
\begin{eqnarray*}
\int_{M}|\mathring{Ric}|^{2}|\nabla f|^{2}dM_{g}&=&\frac{n-3}{2(n-1)}\int_{M}f^{2}|C_{ijk}|^{2}dM_{g}+\int_{M}f^{2}|\nabla Ric|^{2} dM_{g}\\
&&+\frac{n}{n-1}\int_{M}f|\mathring{Ric}|^{2}dM_{g}+\frac{2}{n-1}\int_{M}f^{2}R|\mathring{Ric}|^{2} dM_{g}\\
&&+\int_{M}f^{2}\left(\frac{n}{n-2}tr(\mathring{Ric}^{3})-W_{ijkl}\mathring{R}_{ik}\mathring{R}_{jl}\right) dM_{g}\\
&&-\frac{n-2}{n-1}\int_{M}fC_{ijk}W_{ijkl}\nabla_{l}f dM_{g},
\end{eqnarray*}
where $\mathring{Ric}^{3}$ is the $2$-tensor defined by $(\mathring{Ric}^{3})_{ij}=\mathring{R}_{ik}\mathring{R}_{kl}\mathring{R}_{lj}.$
\end{lemma}
\begin{proof}
Using \cite[Lemma 3.1]{balt18} it is immediate to obtain the following integral identity
\begin{eqnarray}\label{auxI1}
-\int_{M}\langle\nabla f^{2},\nabla|Ric|^{2}\rangle dM_{g}&=&-\int_{M}f^{2}|C_{ijk}|^{2}dM_{g}+2\int_{M}f^{2}|\nabla Ric|^{2}dM_{g}\nonumber\\
&&+\frac{2n}{n-2}\int_{M}f^{2}tr(\mathring{Ric}^{3})dM_{g}+\frac{2}{n-1}\int f^{2}R|\mathring{Ric}|^{2}dM_{g}\nonumber\\
&&+4\int_{M}fC_{ijk}\nabla_{j}fR_{ik}dM_{g}-2\int_{M}f^{2}W_{ijkl}\mathring{R}_{ik}\mathring{R}_{jl}dM_{g},
\end{eqnarray}
where we have used that
$$tr(Ric^{3})=tr(\mathring{Ric}^{3})+\frac{3}{n}R|\mathring{Ric}|^{2}+\frac{1}{n^{2}}R^{3}.$$
On the other hand, using (\ref{auxC}) and (\ref{weyl}), it is not difficult to verify that 
\begin{equation}\label{CWgradf}
\int_{M}fC_{ijk}\nabla_{j}fR_{ik}dM_{g}=\frac{n-2}{2(n-1)}\int_{M}[f^{2}|C_{ijk}|^{2}-fC_{ijk}W_{ijkl}\nabla_{l}f] dM_{g}.
\end{equation}
Moreover, since $M^{n}$ has constant scalar curvature we may use (\ref{eqtrace})  to deduce 
\begin{eqnarray}\label{auxDiv}
{\rm div}(|\mathring{Ric}|^{2}\nabla f^{2})-\langle\nabla f^{2},\nabla|Ric|^{2}\rangle&=&2f\Delta f|\mathring{Ric}|^{2}+2|\mathring{Ric}|^{2}|\nabla f|^{2}\nonumber\\
&=&-\frac{2}{n-1}f^{2}R|\mathring{Ric}|^{2}-\frac{2n}{n-1}f|\mathring{Ric}|^{2}+2|\mathring{Ric}|^{2}|\nabla f|^{2}.
\end{eqnarray}
As consequence, upon integrating (\ref{auxDiv}) over $M$ we apply the divergence theorem to arrive at
\begin{eqnarray}\label{auxDiv1}
-\int_{M}\langle\nabla f^{2},\nabla|Ric|^{2}\rangle dM_{g}&=&-\frac{2}{n-1}\int_{M}f^{2}R|\mathring{Ric}|^{2}dM_{2}-\frac{2n}{n-1}\int_{M}f|\mathring{Ric}|^{2}dM_{g}\nonumber\\
&&+2\int_{M}|\mathring{Ric}|^{2}|\nabla f|^{2}dM_{g}.
\end{eqnarray}

Finally, just replace (\ref{CWgradf}) and (\ref{auxDiv1}) into (\ref{auxI1}) to get the desired result. 
\end{proof}

Now, we use \cite[Lemma 3.2]{balt18}, to get another formula for $\int_{M}|\mathring{Ric}|^{2}|\nabla f|^{2}dM_{g}.$

\begin{lemma}\label{auxint2}
Let $(M^{n},g,f)$ be a Miao-Tam critical metric. Then we have
\begin{eqnarray*}
\frac{3}{2}\int_{M}|\mathring{Ric}|^{2}|\nabla f|^{2}dM_{g}&=&-\frac{1}{n-1}\int_{M}f^{2}|C_{ijk}|^{2}dM_{g}+\int_{M}f^{2}|\nabla Ric|^{2} dM_{g}\\
&&+\frac{n}{2(n-1)}\int_{M}f|\mathring{Ric}|^{2}dM_{g}+\frac{3}{2(n-1)}\int_{M}f^{2}R|\mathring{Ric}|^{2} dM_{g}\\
&&-\frac{n-2}{n-1}\int_{M}fC_{ijk}W_{ijkl}\nabla_{l}f dM_{g}.
\end{eqnarray*}
\end{lemma}
\begin{proof}
By direct computation using \cite[Lemma 3.2]{balt18}, we achieve
\begin{eqnarray*}
-\frac{3}{4}\int_{M}\langle\nabla f^{2},\nabla|Ric|^{2}\rangle dM_{g}&=&-\int_{M}f^{2}|C_{ijk}|^{2}dM_{g}+\int_{M}f^{2}|\nabla Ric|^{2}dM_{g}\\
&&-\frac{n}{n-1}\int_{M}f|\mathring{Ric}|^{2}dM_{g}+2\int_{M}fC_{ijk}\nabla_{j}fR_{ik}dM_{g}.
\end{eqnarray*}
To conclude, just consider the same argument which was used in the last part of the proof of Lemma~\ref{auxint1}.
\end{proof}

\section{Miao Tam critical metrics with pinched curvature}
In this section we will prove the Theorem~\ref{Wineq} announced in the introduction.

\subsection{Proof of Theorem~\ref{Wineq}} 
First of all, we take the difference between the expressions obtained in Lemma~\ref{auxint1} and Lemma~\ref{auxint2} to infer
\begin{eqnarray}\label{keyID}
0&=&\int_{M}|\mathring{Ric}|^{2}|\nabla f|^{2}dM_{g}+\int_{M}f^{2}|C_{ijk}|^{2}dM_{g}+\frac{n}{n-1}\int_{M}f|\mathring{Ric}|^{2}dM_{g}\nonumber\\
&&+\frac{1}{n-1}\int_{M}f^{2}R|\mathring{Ric}|^{2}dM_{g}+2\int_{M}f^{2}\left(\frac{n}{n-2}tr(\mathring{Ric}^{3})-W_{ijkl}\mathring{R}_{ik}\mathring{R}_{jl}\right) dM_{g}.
\end{eqnarray}
Before proceeding it is important to remember the following inequality
\begin{eqnarray*}
\left|\frac{n}{n-2}\mathring{R}_{ij}\mathring{R}_{jk}\mathring{R}_{ik}-W_{ijkl}\mathring{R}_{ik}\mathring{R}_{jl}\right|\leq\sqrt{\frac{n-2}{2(n-1)}}\left(|W|^{2}+\frac{2n}{n-2}|\mathring{Ric}|^{2}\right)^{1/2}|\mathring{Ric}|^{2},
\end{eqnarray*} 
which is true for an arbitrary $n$-dimensional Riemannian manifold. Its proof can be found in  \cite{BaltCPE} or \cite{Fu17}. 

Hence, using the above inequality and our curvature estimate (\ref{Welyestimate}), it is immediate to check that
\begin{eqnarray*}
&&\frac{1}{n-1}R|\mathring{Ric}|^{2}+2\left(\frac{n}{n-2}tr(\mathring{Ric}^{3})-W_{ijkl}\mathring{R}_{ik}\mathring{R}_{jl}\right)\\
&&\geq\left[\frac{1}{n-1}R-2\sqrt{\frac{n-2}{2(n-1)}}\left|W+\frac{\sqrt{n}}{\sqrt{2}(n-2)}\mathring{Ric}\varowedge g\right|\right]|\mathring{Ric}|^{2}\geq0.
\end{eqnarray*}
As consequence, returning to Identity (\ref{keyID}), we force $M^{n}$ to be Einstein manifold. Then, we are in position to use Theorem 1.1 of \cite{MT11} to conclude that $(M^{n},g)$ is isometric to a geodesic ball in $\mathbb{S}^{n}.$ So, the proof is completed.

\section{Proof of Theorem~\ref{MTf}}

\subsection{Proof of Theorem~\ref{MTf}}
To start with, taking into account that $M$ has constant scalar curvature, we derive our hypothesis on both side to infer
$$\nabla_{i}\nabla_{j}f\nabla_{j}f+\frac{Rf}{n(n-1)}\nabla_{i}f+\frac{1}{n-1}\nabla_{i}f=0.$$
Hence, by $(\ref{eq:tensorial})$, we have 
$$fR_{ij}\nabla_{j}f+\left(\Delta f+1+\frac{Rf}{n(n-1)}+\frac{1}{n-1}\right)\nabla_{j}f=0,$$
which can be written using (\ref{eqtrace}) as
\begin{equation}\label{Ricnablaf}
Ric(\nabla f)=\frac{R}{n}\nabla f.
\end{equation}
Here, we have used that $f$ is positive in the interior of $M$ and the equality immediately follows by continuity. 

On the other hand, using again our hypothesis we may conclude the following expression  of the laplacian on the norm of the gradient of the potential function $f$,
\begin{eqnarray}\label{auxLap}
\frac{1}{2}\Delta|\nabla f|^{2}&=&-\frac{R}{2n(n-1)}\Delta (f^{2})-\frac{1}{n-1}\Delta f\nonumber\\
&=&-\frac{Rf+n}{n(n-1)}\Delta f-\frac{R}{n(n-1)}|\nabla f|^{2}\nonumber\\
&=&\frac{1}{n}(\Delta f)^{2}-\frac{R}{n(n-1)}|\nabla f|^{2}.
\end{eqnarray}

Consequently, combining the classical Bochner's formula \cite{Bochner}
$$\frac{1}{2}\Delta|\nabla f|^{2}=Ric(\nabla f, \nabla f)+\langle\nabla\Delta f,\nabla f\rangle+|Hessf|^{2},$$
Equalities (\ref{auxLap}) and (\ref{eqtrace}), we arrive at
\begin{eqnarray*}
\frac{1}{n}(\Delta f)^{2}-\frac{R}{n(n-1)}|\nabla f|^{2}&=&Ric(\nabla f,\nabla f)+\left\langle\nabla\left(\frac{-Rf-n}{n-1}\right),\nabla f\right\rangle+|Hessf|^{2}\\
&=&Ric(\nabla f,\nabla f)-\frac{R}{n-1}|\nabla f|^{2}+|\mathring{Hessf}|^{2}+\frac{1}{n}(\Delta f)^{2},
\end{eqnarray*}
that is, 
$$|\mathring{Hessf}|^{2}+\mathring{Ric}(\nabla f, \nabla f)=0.$$

Finally, using (\ref{Ricnablaf}), the last equality allow us to conclude that the $\mathring{Hessf}$ is identically zero in $M^{n}.$ This implies, from (\ref{IdRicHess}),  that the $M$ is an Einstein manifold and we are in position to use again Theorem 1.1 of \cite{MT11}  to conclude that $(M^{n},g)$ is isometric to a geodesic ball in a simply connected space form $\mathbb{R}^{n},$ $\mathbb{H}^{n}$ or $\mathbb{S}^{n}$.

\begin{acknowledgement}
The first author was partially supported by PPP/FAPEPI/MCT/CNPq, Brazil [Grant: 007/2018] and CNPq/Brazil [Grant: 422900/2021-4].
\end{acknowledgement}

\end{document}